\documentclass{amsart}

\usepackage{geometry} 
\newcommand{\PI}{3.14159}
\usepackage{graphicx} 
\usepackage{amssymb}
\usepackage{amsmath}
\usepackage{amsthm}
\usepackage{tikz}
\usepackage{hyperref}
\usepackage{mathrsfs}
\newtheorem{theorem}{Theorem}[section]

\newtheorem{proposition}[theorem]{Proposition}
\newtheorem{conjecture}[theorem]{Conjecture}

\newtheorem{question}[theorem]{Question}

\theoremstyle{definition}
\newtheorem{definition}[theorem]{Definition}
\newtheorem{remark}[theorem]{Remark}
\newtheorem{example}[theorem]{Example}
\DeclareMathOperator{\flag}{fdes}
\usetikzlibrary{patterns}
\usepackage{color}
\begin{document}
\title[Flag descents and eulerian polynomials]{Flag descents and Eulerian polynomials for wreath product quotients}
\author{Dustin Hedmark}
\address{Department of Mathematics\\
         University of Kentucky\\
         Lexington, KY 40506--0027}
\email{{dustin.hedmark@uky.edu}}

\author{Cyrus Hettle} 
\address{Department of Mathematics\\
         University of Kentucky\\
         Lexington, KY 40506--0027}
\email{\href{mailto:cyrus.h@uky.edu}{cyrus.h@uky.edu}}

\author{McCabe Olsen}
\address{Department of Mathematics\\
         University of Kentucky\\
         Lexington, KY 40506--0027}
\email{\href{mailto:mccabe.olsen@uky.edu}{mccabe.olsen@uky.edu}}
\date{\today}
\thanks{The third author was partially supported by NSF-OEIS 1613525.}
\begin{abstract}
We investigate the $\alpha$-colored Eulerian polynomials and a notion of descents introduced in a recent paper of Hedmark and show that such polynomials can be computed as a polynomial encoding descents  computed over a quotient of the wreath product $\mathbb{Z}_\alpha\wr\mathfrak{S}_n$. Moreover, we consider the flag descent statistic computed over this same quotient and find that the flag Eulerian polynomial remains palindromic. 
We prove that the flag descent polynomial is palindromic over this same quotient by giving a combinatorial proof that the flag descent statistic is symmetrically distributed over the collection of colored permutations with fixed last color by way of a new combinatorial tool, the colored winding number of a colored permutation. We conclude with some conjectures, observations, and open questions. 
\end{abstract}

\maketitle
\tableofcontents
\section{Introduction}
Given $\alpha\in \mathbb{Z}_{>0}$, let $\mathbb{Z}_\alpha$ denote the cyclic group of order $\alpha$. Let $\mathfrak{S}_n$ be the symmetric group on $n$~elements. Recall that the \emph{Eulerian polynomial} is the polynomial
\[
A_n(x)=\sum_{\pi\in \mathfrak{S}_n}x^{{\rm des}(\pi)}
\]
where ${\rm des}(\pi)=|\{i \, : \, \pi_i>\pi_{i+1} \}|$. This polynomial also arises as the $h$-polynomial of the permutahedron~\cite{Petersen}.

In a recent paper~\cite{hedmark}, the first author introduced a  new polytopal complex called the \emph{$\alpha$-colored permutahedron} $P^\alpha_n$, as well as a new notion of descent for colored permutations in the wreath product $\mathbb{Z}_\alpha\wr\mathfrak{S}_n$. The \emph{$\alpha$-colored Eulerian polynomials} $A^\alpha_n(x)$ are a generalization of the usual Eulerian polynomial using these descents computed over  colored permutations with fixed last color. These polynomials are equivalently defined as the $h$-polynomial of $P^\alpha_n$. 


In this paper, we extend the work of~\cite{hedmark} by providing a new algebraic interpretation of~$A_n^\alpha$. In particular, we show that these polynomials arise by computing descent polynomials over the quotient $\mathbb{Z}_\alpha\wr \mathfrak{S}_n/H$, where $H$ is an appropriately chosen subgroup of order $\alpha$. Moreover, we consider additional descent statistics computed over the quotient $\mathbb{Z}_\alpha\wr\mathfrak{S}_n/H$, namely the \emph{flag descent statistic}, originally introduced for the  hyperoctahedral group $B_n\cong \mathbb{Z}_2\wr \mathfrak{S}_n$ in~\cite{hyperoctahedral} and generalized for $\mathbb{Z}_\alpha\wr \mathfrak{S}_n$ for $\alpha\geq 2$ in~\cite{bagnobiagoli}. We give a combinatorial proof that the flag descent statistic is symmetrically distributed over coset representatives of $\mathbb{Z}_\alpha\wr\mathfrak{S}_n/H$ by introducing the notion of the \emph{colored winding number} of a colored permutation, defined in Section~\ref{section_flag}.

We end the paper with a collection of conjectures, observations and open questions. Specifically, we make an overarching conjecture of unimodality, as well as more specific conjectures and observations for the particular case of $\alpha=2$ and $n=2k+1$, citing connections to discrete geometry and other enumerative methods. Finally, we state three tractable open questions towards showing the conjecture. 


\section{Generalized symmetric groups and  
colored Eulerian polynomials}
\label{generalized_perm_groups}
The \emph{generalized symmetric group} is the wreath product $\mathbb{Z}_\alpha\wr\mathfrak{S}_n$, or the group given by the semidirect product $(\mathbb{Z}_\alpha)^n\ltimes\mathfrak{S}_n$. Note that if $\alpha=1$ this coincides with the usual symmetric group. We denote an element of the generalized symmetric group by $w=w_1^{c_1}w_2^{c_2}\cdots w_n^{c_n}$, where $w_1w_2\cdots w_n$ is in $\mathfrak{S}_n$ and each color $c_i \in \mathbb{Z}_\alpha$. We now record the action of the generalized symmetric group in a definition.

\begin{definition}
\label{grp_action}
Let $w=w_1^{c_1}w_2^{c_2}\cdots w_n^{c_n}$ and $y=y_1^{d_1}y_2^{d_2}\cdots y_n^{d_n}$ in $\mathbb{Z}_\alpha\wr\mathfrak{S}_n$. Their product is
$$w_1^{c_1}w_2^{c_2}\cdots w_n^{c_n}\cdot y_1^{d_1}y_2^{d_2}\cdots y_n^{d_n}=w_{y_1}^{d_{1}+c_{y_1}}w_{y_2}^{d_2+c_{y_{2}}}\cdots w_{y_n}^{d_n+c_{y_{n}}},$$
where the exponents are taken modulo $\alpha$. In other words, we multiply the underlying permutations and assign the $i$th color by adding $d_i$ to the action of $y$ on $c_i$.
\end{definition}
\begin{remark}
\label{matrix}
The group multiplication in  $\mathbb{Z}_{\alpha}\wr\mathfrak{S}_n$ coincides with the multiplication of permutation matrices. Explicitly, using the colored permutation $w$ of Definition~\ref{grp_action}, we have that $w$ corresponds to a permutation matrix $A$ with $a_{i,j}=\zeta^{c_j}$ if $w_j=i$, for $\zeta=e^{2\pi i/\alpha}$ a root of unity, and $0$ otherwise.
\end{remark}

We will now show that the $\alpha$-colored Eulerian polynomials of ~\cite{hedmark} can be computed over the quotient $\mathbb{Z}_\alpha\wr\mathfrak{S}_n/H$, where $H$ the cyclic subgroup generated by $w=1^{1}2^1\cdots n^1$, that is, the identity permutation with all colors $1$. We recall the following definitions from~\cite{hedmark}.
\begin{definition}
Let $\mathbb{Z}_\alpha\wr\mathfrak{S}^{\beta}_n$ be the collection of colored permutations in $\mathbb{Z}\wr\mathfrak{S}_n$ with fixed last color $c_{n}=\beta$.
\end{definition}
\begin{definition}
Given  $w=w_1^{c_1}w_2^{c_2}\cdots w_n^{c_n}\in \mathbb{Z}_\alpha\wr\mathfrak{S}^{\beta}_n$, let the colored descent set $D(w)$ be $\{i \, : c_i\neq c_{i+1} \mbox{ or } c_i=c_{i+1} \mbox{ and }  w_i>w_{i+1}\}$ and denote the number of colored descents $|D(w)|$ by $d(w)$. 
The $\alpha$-colored Eulerian polynomials are defined as 
	\[
	A_n^{\alpha}(x)=\sum_{w\in\mathbb{Z}_\alpha\wr\mathfrak{S}^{\beta}_n}x^{d(w)}.
    \]
\end{definition}
By convention, we will assume that $\beta=0$, as the definition assigns no ordering to the colors. In the case of $\alpha=1$, this is the classical Eulerian polynomial $A_n(x)$.

Note that this notion of colored descent differs from those given by Steingr\'{i}mmson in~\cite{Steingrimmson_indexed} and by Ehrenborg and Readdy in~\cite{ER_r-cubical}.

We now show that $A^\alpha_n$ arises algebraically as follows. 

\begin{proposition}
The $\alpha$-colored Eulerian polynomial can be computed as
	\[
	A_n^{\alpha}(x)=\sum_{w\in\mathbb{Z}_\alpha\wr\mathfrak{S}_n/H}x^{d(w)}
    \]
where $H$ is the cyclic subgroup generated by $w=1^{1}2^1\cdots n^1$ in $\mathbb{Z}_{\alpha}\wr\mathfrak{S}_n$.  
\end{proposition}
\begin{proof}
The matrix corresponding to $w$ is the diagonal matrix $\zeta I_{n}$ of Remark~\ref{matrix}, so the subgroup generated by $w$ is the collection of diagonal matrices $\zeta^i I_{n}$ for $0\leq i\leq\alpha-1$. The subgroup $H$ is clearly normal, as it consists of scalar multiples of the identity matrix.

Cosets of $\mathbb{Z}_{\alpha}\wr\mathfrak{S}_n/H$ consist of a given colored permutation $y$ and all cylic shifts of the color vector of $y$. In other words, two colored permutations $y$ and $z$ are in the same coset of $H$ if $y$ and $z$ have the same underlying permutation $w_{1}\cdots w_{n}$ and their color vectors differ by a multiple of $(1,1,\dots,1)$. Notice that if $y$ and $z$ are in the same coset of $H$, then $d(y)=d(z)$, as cyclically shifting colors preserves color changes and leaves the underlying permutation unchanged.

As the descent number is well-defined over cosets of $\mathbb{Z}_{\alpha}\wr\mathfrak{S}_n/H$, by cyclic shifting we may choose a set of representatives such that each representative has last color $0$. Moreover, we obtain every possible colored permutation with last color $0$ as a representative in this way. Therefore, to compute the $\alpha$-colored Eulerian polynomial as above we can instead compute the sum over the quotient $\mathbb{Z}_{\alpha}\wr\mathfrak{S}_n/H$, that is, $A_n^{\alpha}(x)=\sum_{w\in\mathbb{Z}_\alpha\wr\mathfrak{S}_n/H}x^{d(w)}$.
\end{proof}

\section{The flag descent statistic}
\label{section_flag}
In this section, we explore a different notion of descent on $\mathbb{Z}_\alpha\wr\mathfrak{S}_n$, namely the flag descent statistic. We consider generalizations of Eulerian polynomials with respect to this statistic. 
\begin{definition}[Adin--Brenti--Roichman,~\cite{hyperoctahedral}; Bagno--Biagioli,~\cite{bagnobiagoli}]
\label{flag_des}
The {\it flag descent statistic} on $\mathbb{Z}_{\alpha} \wr \mathfrak{S}_{n}$ is given by 
\[
\flag(w)=\alpha\cdot|\{i:c_{i}=c_{i+1},w_{i}>w_{i+1}\}|+\alpha \cdot|\{i:c_{i}<c_{i+1}\}|+c_{1},
\]
where the addition of $c_{1}$ takes place in $\mathbb{Z}$ and the order on $\mathbb{Z}_{\alpha}$ is the linear order $0\prec 1\prec\cdots\prec\alpha-1$.
\end{definition}

\begin{example} In $\mathbb{Z}_{3} \wr \mathfrak{S}_{4}$, we have that $\flag{(4^{1}1^{1}3^{2}2^{0})}=3\cdot1+3\cdot1+1=7$.
\end{example}

The notion of flag descents is motivated by the notion of the {\it flag major index}, which can be defined from the structure of $\mathbb{Z}_\alpha\wr \mathfrak{S}_n$ viewed as a Coxeter group (e.g. see~\cite{flagmajor} for details).  Moreover, when $\alpha=1$, we have $\flag(\pi)=d(\pi)$, where $d(\pi)$ is the classical descent statistic.

We now define the \emph{flag Eulerian} polynomial in the obvious way.

\begin{definition}[Foata--Han,\cite{Foata}]
The \emph{flag Eulerian} polynomial over $\mathbb{Z}_\alpha\wr\mathfrak{S}_n$ is given by the sum 

\[\mathcal{W}_n^\alpha(x)=\sum_{\pi\in\mathbb{Z}_{\alpha}\wr\mathfrak{S}_n}x^{\flag(\pi)}.
\]
\end{definition}
\noindent The flag Eulerian polynomial is palindromic  which follows from Proposition 2.3 of~\cite{edgewise}. However, we are interested in the flag Eulerian polynomials computed over the quotient $\mathbb{Z}_\alpha\wr\mathfrak{S}_n/H$, which we define as
	\begin{equation}\label{Flag_Eulerian_Quotient_Poly}
  \mathcal{F}_n^\alpha(x)=\sum_{\pi\in\mathbb{Z}_{\alpha}\wr\mathfrak{S}_n/H}x^{\flag(\pi)}
    \end{equation} 
where we choose the coset representative with last color 0. We note that while there are known recurrences for the flag descent numbers over $\mathbb{Z}_\alpha\wr \mathfrak{S}_n$~\cite{Foata}, these recurrences do not adapt in an obvious way to elements with fixed last color. 
 
We now give a combinatorial proof that the polynomial $\mathcal{F}_n^\alpha(x)$ is palindromic. This polynomial has degree $\alpha\cdot(n-1)$, as the colored permutation $\pi=n^0(n-1)^0\cdots 1^0$ has maximal flag descent number $\flag(\pi)=\alpha\cdot(n-1)$. 
\begin{theorem} 
\label{main_result}
The flag descent statistic $\flag$ is symmetric about $\alpha\cdot(n-1)/2$ over $\mathbb{Z}_\alpha~\wr~\mathfrak{S}_{n}^{0}$. As a consequence $\mathcal{F}_n^\alpha(x)$ is a palindromic polynomial. 
\end{theorem}\label{symmetric}
\begin{proof}
To prove the symmetry of $\mathcal{F}_n^\alpha(x)$, we exhibit a bijection $r$ on the permutations in $\mathbb{Z}_\alpha\wr\mathfrak{S}_n^0$ satisfying  
\begin{equation}
\label{r_map_symmetric_equation}
\flag(w)+\flag(r(w))=\alpha\cdot(n-1).
\end{equation}
In the usual symmetric group, the reversal map applied to permutations in one-line notation proves that the Eulerian polynomials are palindromic. To show the $\mathcal{F}_n^\alpha(x)$ are palindromic, we will reverse the underlying permutations and cycle colors to guarantee the last color is $0$. That is, 
define a map $r: \mathbb{Z}_\alpha\wr\mathfrak{S}_{n}^{0} \longrightarrow \mathbb{Z}_\alpha\wr\mathfrak{S}_{n}^{0}$ by
$r(w_{1}^{c_{1}}\cdots w_{n}^{c_{n}})=w_{n}^{c_{n}-c_{1}}w_{n-1}^{c_{n-1}-c_{1}}\cdots w_{1}^{0}$, where the subtraction of $c_{1}$ takes place in $\mathbb{Z}_{\alpha}$. We claim that that equation~\eqref{r_map_symmetric_equation} holds for all $w \in \mathbb{Z}_\alpha\wr\mathfrak{S}_{n}^{0}$. Since $r$ is clearly bijective, this suffices to prove symmetry.

First, we reduce the problem to proving that equation~\eqref{r_map_symmetric_equation} holds for elements of $\mathbb{Z}_\alpha\wr\mathfrak{S}_n^0$ with no two equal adjacent colors. Given a colored permutation with $c_{i}=c_{i+1}$ and $w_{i}>w_{i+1}$, so that there is a descent at position $i$, we can remove that descent by deleting $w_{i}^{c_{i}}$ and relabeling $w'=w_{1}\cdots\widehat{w_{i}}\cdots w_{n}$ appropriately to obtain an element of $\mathbb{Z}_\alpha\wr\mathfrak{S}_{n-1}^{0}$ such that either $\flag(w')=\flag(w)-\alpha$ and $\flag(r(w'))=\flag(r(w))$ or $\flag(w')=\flag(w)$ and $\flag(r(w'))=\flag(r(w))-\alpha$. In either case, induction on $n$ gives that $\flag(w')+\flag(r(w'))=\alpha\cdot(n-2)$, and thus $\flag(w)+\flag(r(w))=\alpha\cdot(n-1)$ by the equalities of the previous sentence.

Since we have reduced to the case where no adjacent colors are equal, Definition~\ref{flag_des} reduces to \[\flag(w)=\alpha\cdot|\{i:c_i<c_{i+1}\}|+c_1.\]
With this form of the flag descent statistic, a routine calculation shows that establishing equation~\eqref{r_map_symmetric_equation} is equivalent to showing that
\begin{equation}
\label{second_r_symmetric_equation}
|\{i:c_{i}<c_{i+1}\}|+|\{i:c_{i+1}-c_{1}<c_{i}-c_{1}\}|=n-2.
\end{equation}
when $c_1\neq 0$ and 
\begin{equation}
\label{c1_zero_r_symmetric_equation}
|\{i:c_{i}<c_{i+1}\}|+|\{i:c_{i+1}<c_{i}\}|=n-1
\end{equation}
when $c_1=0$.

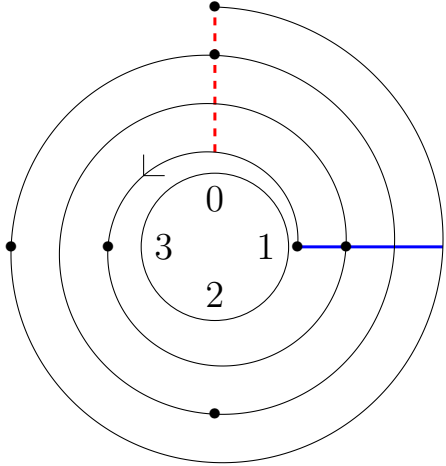
\begin{figure}[h!]
\label{clock_picture_counterclockwise}
\begin{center}
\begin{tikzpicture}[scale=.7]
\draw [color=red, dashed, line width=0.4mm](0,1.8) -- (0,4.56);
\draw [color=blue, line width=0.4mm](1.57,0) -- (4.33,0);
\node at (0,1.5)[label=below:\LARGE$0$]{};
\node at (1.5,0)[label=left:\LARGE$1$]{};
\node at (0,-1.5)[label=above:\LARGE$2$]{};
\node at (-1.5,0)[label=right:\LARGE$3$]{};
\node (1) at (1.57,0){$\bullet$};
\node (2) at (-2.03,0){$\bullet$};
\node (3) at (2.49,0){$\bullet$};
\node (4) at (0,-3.18){$\bullet$};
\node (5) at (0,3.64){$\bullet$};
\node (6) at (-3.87,0){$\bullet$};
\node (7) at (0,4.56){$\bullet$};

\draw (-1.354,1.754)--(-1.354,1.354)--(-.954,1.354);
\draw (0,0) circle (1.4cm);
 \draw [domain=0:(13*\PI)/2,variable=\t,smooth,samples=400]
        plot ({\t r}: {1.8+(2.3/(5*\PI))*(\t-\PI/2)});
\end{tikzpicture}
\caption{We compute the colored winding numbers $W((1,3,1,2,0,3,0),0)$ and $W((1,3,1,2,0,3,0),1)$ with colors in $\mathbb{Z}_4$. We begin at the first color $1$ and traverse the spiral from the inside out. The hand crosses the red dashed line and blue solid line four times each, so $W((1,3,1,2,0,3,0),0)=W((1,3,1,2,0,3,0),1)=3$.} 
\label{colored_winding_forward}
\end{center}
\end{figure}

We examine the cases where $c_1=0$ and when $c_1\neq 0$ separately.

\textbf{Case 1:} $c_{1}=0$. 

As our colored permutations $w$ have no equal adjacent colors, at each index $i$ we must have that $c_i<c_{i+1}$ or $c_{i+1}<c_i$, and thus equation~\eqref{c1_zero_r_symmetric_equation} holds.

\textbf{Case 2:} $c_{1}\neq0$. 

We show equation~\eqref{second_r_symmetric_equation} via a combinatorial interpretation. Consider a clock with a single hand and the $\alpha$ congruence classes $0, 1,\dots,\alpha-1$ marked on it clockwise in the usual cyclic order. Let $w_{1}^{c_{1}}\cdots w_{n}^{c_{n}}\in \mathbb{Z}_\alpha\wr\mathfrak{S}_n^0$ be a colored permutation such that $c_{i}\neq c_{i+1}$ for all $i$. With the hand starting at $c_{1}$, rotate it counterclockwise until it reaches $c_{2}$, then rotate it counterclockwise to $c_{3}$, and so forth until it has reached $c_{n}=0$. Now count the number of times it has rotated to or past $i \in \mathbb{Z}_\alpha$ and subtract 1. This is the \textit{colored winding number} $W(w,i)$ of the sequence $(c_{1},\cdots,c_{n})$. (Since the winding number depends only on the sequence of colors, we use $W(w,i)$ and $W((c_{1},c_{2},\cdots,c_{n}),i)$ interchangeably.) See Figure~\ref{colored_winding_forward} for an example.

The path traced out by the hand consists of an arc, or a spiral, see Figure~\ref{colored_winding_forward}, from $c_{1}$ to $0$. It is easy to see that the hand rotates to or past $0$ the same number of times that it rotates to or past~$c_{1}$, so $W(w,0)=W(w,c_{1})$. 

In addition, since there is a descent every time the hand passes 0, and every time the hand touches 0 (except the last time) there is a descent at the next step, the colored winding number counts colored descents. That is,
\begin{equation}
\label{winding_number_counts_color_descents_equation}
W(w,0)=|\{i:c_{i}<c_{i+1}\}|.
\end{equation}
For example, the colored winding number $W(w,0)$ of the sequence $(1, 3, 1, 2, 0, 3, 0)$ is 3: the hand passes 0 when it goes to 3, passes 0 when it goes to 2, touches 0, then touches 0 one more time. The hand is also at or passes 1 four times (it starts at 1, touches 1 again, passes 1 when it goes to 0 the first time, and passes 1 when it goes to 0 at the end) and the sequence increases three times: $1<3,1<2,0<3$. This is illustrated in Figure~\ref{colored_winding_forward}.

Now consider the colored winding number of $r(w_{1}^{c_{1}}\cdots w_{n}^{c_{n}})$, or the color sequence $(c_{n}-c_{1},c_{n-1}-c_{1},\dots,0)$, where we rotate the hand clockwise instead of counterclockwise. We denote this winding number by $W'(r(w),0)$. We compute $W'$ for the color vector $(1,3,1,2,0,3,0)$ in Figure~2. 
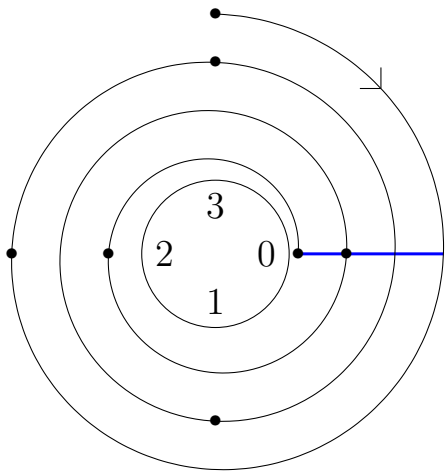
\begin{figure}[h]
\begin{center}
\begin{tikzpicture}[scale=.7] 
\draw [color=blue, line width=0.4mm](1.57,0) -- (4.33,0);
\node at (0,1.5)[label=below:\LARGE$3$]{};
\node at (1.5,0)[label=left:\LARGE$0$]{};
\node at (0,-1.5)[label=above:\LARGE$1$]{};
\node at (-1.5,0)[label=right:\LARGE$2$]{};
\node (1) at (1.57,0){$\bullet$};
\node (2) at (-2.03,0){$\bullet$};
\node (3) at (2.49,0){$\bullet$};
\node (4) at (0,-3.18){$\bullet$};
\node (5) at (0,3.64){$\bullet$};
\node (6) at (-3.87,0){$\bullet$};
\node (7) at (0,4.56){$\bullet$};

\draw (2.743,3.143)--(3.143,3.143);
\draw (3.143,3.143)--(3.143,3.543);
\draw (0,0) circle (1.4cm);
 \draw [domain=0:(13*\PI)/2,variable=\t,smooth,samples=200]
        plot ({\t r}: {1.8+(2.3/(5*\PI))*(\t-\PI/2)});
    
\end{tikzpicture}
\label{winding_backwards}
\caption{We compute the colored winding number $W'(r(1,3,1,2,0,3,0),0)=W'((3,2,3,1,0,2,0),0)$ with colors in $\mathbb{Z}_4$. The path taken by the hand is identical to the path taken in computing $W((1,3,1,2,0,3,0),1)$. Notice that we now begin at first color $3$ and traverse the spiral outside in. Additionally, notice the clock has been rotated by $1$ clockwise from the clock in Figure~\ref{colored_winding_forward}, see the paragraph preceding equation~\eqref{winding_reverse_winding_equality}.} 
\end{center}
\end{figure}

By equation~\eqref{winding_number_counts_color_descents_equation} we have that 
\[
W'(r(w),0)+1=|\{i:c_{i+1}-c_{1}>c_{i}-c_{1}\}|=n-1-|\{i:c_{i+1}-c_{1}<c_{i}-c_{1}\}|.
\]
Moreover, rotating the hand clockwise and taking the numbers in opposite order is exactly the procedure for computing $W(w,0)$ in reverse, except that we have introduced a shift, which is equivalent to rotating the labels on the clock clockwise by $c_{1}$. Therefore $W'(r(w),0)=W(w,c_{1})$. 

Thus we have that
\begin{equation}
\label{winding_reverse_winding_equality}
n-2-|\{i:c_{i+1}-c_{1}<c_{i}-c_{1}\}|=W'(r(w),0)=W(w,c_{1})=W(w,0)=|\{i:c_{i}<c_{i+1}\}|
\end{equation}
establishing equation~\eqref{second_r_symmetric_equation} and finishing the proof.
\end{proof}

\section{Concluding remarks and future directions}

We conclude this note with a collection of conjectures, observations, and open questions. In addition to the symmetric distribution of the flag descent statistic over the quotient $\mathbb{Z}_{\alpha}\wr \mathfrak{S}_n/H$, computational evidence motivates the following conjecture:
\begin{conjecture}\label{unimodality_conjecture}  The polynomial $\mathcal{F}_n^\alpha(x)$ is unimodal and palindromic.
\end{conjecture}
\noindent Symmetry follows immediately from Theorem~\ref{main_result}. However, standard techniques for showing unimodality have not proven fruitful in the general setting. In certain cases, these polynomials indeed appear to be real-rooted. 

In the specific case of $\alpha=2$ and $n=2k+1$, we have observed that 
\begin{equation}
\label{concluding_remarks_conjecture}
\mathcal{F}_{2k+1}^2(x)=(1+x)^{2k}\cdot A_{2k+1}(x).
\end{equation}
When $\alpha=2$, the generalized symmetric group  is known as the \emph{hyperoctahedral group} $B_n\cong\mathbb{Z}_2\wr\mathfrak{S}_n$. In~\cite{hyperoctahedral}, Adin, Brenti, and Roichman prove that the flag Eulerian polynomial $\mathcal{W}_n^2(x)$ computed over the entire hyperoctahedral group, with no restriction on the last color, satisfies the identity 
	\begin{equation}
    \mathcal{W}_n^2(x)=(1+x)^n \cdot A_n(x)
    \end{equation}
which bears a striking resemblance to  our observed equation~\eqref{concluding_remarks_conjecture}.

These observations lead to several open questions.
\begin{question}\label{question_2_odd}
Is there a combinatorial proof of equation  \eqref{concluding_remarks_conjecture}? In the absence of such a proof, is there a natural inductive proof?
\end{question}
For a combinatorial proof, the barred permutations, or ``balls in boxes", method of Petersen~\cite{Petersen_balls_boxes} may be a fruitful approach. By the work of Beck, Jayawant, and McAllister~\cite{Beck-FreeSum}, equation \eqref{concluding_remarks_conjecture} arises as the $h^{*}$-polynomial of the polytope obtained by the \emph{free sum} $[0,1]^{2k+1}\oplus \beta_{2k}$, where $\beta_{2k}$ is the cross polytope of dimension $2k$. Perhaps this Ehrhart theoretic connection may lead to a discrete geometric proof of the identity. 

A purely inductive proof of this identity has proven elusive due to the lack of a recurrence formula for flag descents in this quotient. Hence we ask the following:

\begin{question}
Is there a recurrence formula for $F(n,\alpha,k)=|\{\pi\in\mathbb{Z}_\alpha\wr\mathfrak{S}_n/H \, : \, \flag(\pi)=k \}|$? Are there recurrences for specific values of $\alpha$?
\end{question}
Recurrence formulas exist for the these values over the entire wreath product $\mathbb{Z}_\alpha\wr\mathfrak{S}_n$ (see~\cite{Foata}), but these formulas break down for computation over the quotient. A natural starting place would be to focus on the case of $\alpha=2$, as a formula in this case may lead to an inductive proof of equation~\eqref{concluding_remarks_conjecture} and new insight for Question~\ref{question_2_odd}.

We conclude with the following question:
\begin{question}
Can unimodality be proven via real-rootedness techniques for certain $\alpha$ to yield partial results for Conjecture~\ref{unimodality_conjecture}?
\end{question}
Showing unimodality via real-rootedness is a common technique in recent combinatorics (see e.g. ~\cite{Branden-realrooted,real_rooted_lecturehall}). In certain cases, we have observed these polynomials to be real-rooted. As such, these techniques may prove useful towards proving unimodality.

\bibliographystyle{plain}
\bibliography{colored_winding}

\end{document}